\documentclass[a4paper,10pt]{article}
\usepackage[english]{babel}
\usepackage{amsthm,amsmath,amssymb,bbm}
\usepackage{stackengine}
\usepackage{natbib}
\usepackage{graphicx}

\newtheorem{thm}{Theorem}[section]

\newtheorem{lemma}[thm]{Lemma}

\newtheorem{defin}[thm]{Definition}

\newtheorem{corollary}[thm]{Corollary}

\def\P{{\mathbb{P}}}

\def\N{{\mathbb{N}}}

\def\limt{\lim_{t\to\infty}}
\def\limt0{\lim_{t\to 0}}

\def\|{\,|\,}

\def\bn#1\en{\begin{align*}#1\end{align*}}
\def\bnn#1\enn{\begin{align}#1\end{align}}

\allowdisplaybreaks
\title{Special Dirichlet Processes: Structure, Uniqueness and Stability}
\author{Philip Kennerberg
\footnote{Philip Kennerberg  
Independent researcher, Sweden  
Email: pkennerberg@gmail.com}
}

\begin{document}
\maketitle

\begin{abstract}
We introduce the class of \emph{Special--Dirichlet processes}, consisting of càdlàg adapted processes admitting a decomposition
\[
X=M+\Gamma,
\]
where \(M\) is a local martingale and \(\Gamma\) is an adapted càdlàg process with vanishing continuous quadratic variation whose jumps are predictable and \(\mathcal F_{s-}\)-measurable.

This class arises naturally from transformations of special semimartingales. Classical results imply that sufficiently regular functions of special semimartingales belong to the broad class of Dirichlet processes. We show that such transformed processes possess substantially more structure: they admit a canonical decomposition in which the predictable jump component is explicitly separated from the martingale component. This yields a refinement of the traditional classification, which previously identified these processes only as Dirichlet processes.

We establish uniqueness of the decomposition and prove that the class is stable under a large family of nonsmooth transformations, including primitives of locally bounded functions with at most countably many discontinuities. An explicit It\^o-type decomposition is obtained in terms of the martingale jump measure and its compensator.

Finally, we investigate stability properties of the canonical decomposition. Under convergence in quadratic variation and Skorokhod \(J_1\)-convergence, we prove stability of both the martingale and singular components after transformation. The proof relies on a threshold isolation principle for jump structures, allowing large jumps to be separated from small-jump contributions and yielding convergence of the transformed decompositions.
\end{abstract}

\section{Introduction}

The class of Dirichlet processes was introduced in
\cite{Fol} as an extension of the class of continuous
semimartingales. In the original framework, a Dirichlet
process admits a decomposition
\[
X=M+A,
\]
where \(M\) is a martingale and \(A\) is a continuous process of
zero energy.

Subsequent developments led to several extensions of the original
Dirichlet framework. In \cite{NonCont}, the Dirichlet framework was extended beyond the
continuous semimartingale setting by considering decompositions
\[
X=Z+C,
\]
where \(Z\) is a semimartingale and \(C\) is continuous with
\([C]=0\). Stability questions under quadratic variation convergence in this setting were later studied
in \cite{Collectanea}. A different direction was pursued in \cite{WeakDir},
which introduced weak Dirichlet processes. More recently,
\cite{Low} considered processes of the form
\[
X=Z+C,
\]
where \(Z\) is a semimartingale and \(C\) is a c\`adl\`ag adapted
process satisfying
\[
[C]^c=0.
\]

This setting provides the natural ambient class for the present work.
In \cite{Low}, It\^o-type decomposition results were obtained for
locally Lipschitz functions of semimartingales and, more generally,
for Dirichlet processes under suitable non-charging conditions on the
continuous quadratic variation.

We adopt the notion of Dirichlet processes in the sense of Lowther,
namely processes of the form
\[
        X = Z + C,
\]
where \(Z\) is a semimartingale and \(C\) is a càdlàg adapted process
satisfying
\[
        [C]^c = 0.
\]
This definition allows for jump components in \(C\) and is therefore
suited to the present setting.

We also recall the classical (CMS/Föllmer-type) definition of
Dirichlet processes, where the residual component is required to be
continuous with zero quadratic variation. In particular, this excludes
jump contributions in the residual.

Finally, we consider weak Dirichlet processes in the sense of
Russo et al., which are defined by a decomposition
\[
        X = M + A,
\]
where \(M\) is a local martingale and \(A\) is orthogonal to all
continuous local martingales.

\medskip

The present work originates from the observation that the usual
Dirichlet classification is too coarse for many processes arising from
transformation theory. In particular, transformations of special
semimartingales are not merely Dirichlet processes in the generalized
sense of \cite{Low}. They possess a finer structure: after
transformation, the martingale and predictable jump contributions can
be separated in a canonical way.

A key motivation for introducing the Special--Dirichlet class is that
it arises naturally as the recipient of such transformation results.
Indeed, we show that if \(Z\) is a special semimartingale and \(f\)
belongs to a broad nonsmooth class, then \(f(Z)\) is a
Special--Dirichlet process. Thus the traditional conclusion that
\(f(Z)\) is Dirichlet can be sharpened: the transformed process admits
a canonical decomposition in which the local martingale part and the
predictable singular component are explicitly separated.

This is the main motivation for introducing the class of
\emph{Special--Dirichlet processes}. These are processes admitting a
decomposition
\[
X=M+\Gamma,
\]
where \(M\) is a local martingale and \(\Gamma\) is a c\`adl\`ag adapted
process with vanishing continuous quadratic variation whose jumps are
carried by a thin predictable set and satisfy
\[
        \Delta\Gamma_s\in\mathcal F_{s-}.
\]
The definition is modeled on the role of special semimartingales within
the class of semimartingales, but it is designed for the Dirichlet
setting: the singular component is allowed to have jumps, while its jump
structure is required to be predictable.

\medskip
\paragraph{Relations between Dirichlet-type classes.}
With these conventions, the following inclusions hold:
\[
\boxed{
\text{Special--Dirichlet}
\;\subset\;
\text{Dirichlet (Lowther)}
\;\subset\;
\text{weak Dirichlet}.
}
\]

The first inclusion follows from the fact that in our framework the
residual component \(\Gamma\) satisfies \([\Gamma]^c=0\), and hence
fits into the Lowther definition with \(Z=M\) and \(C=\Gamma\).
However, since \(\Gamma\) may exhibit jumps, our class is in general
not contained in the classical (CMS) Dirichlet class, where the
residual component is required to be continuous.

\medskip

Dirichlet processes in the sense of \cite{Low} do not admit a canonical
decomposition in general, since the splitting into semimartingale and
zero--continuous--quadratic--variation parts is not unique. By contrast,
weak Dirichlet processes admit a unique decomposition under the
orthogonality condition.

The uniqueness result therefore shows that the Special--Dirichlet
condition identifies a canonical substructure inside the general
Lowther framework: the decomposition is unique, while the singular
component retains an explicit interpretation in terms of predictable
jumps and vanishing continuous quadratic variation.

The class also has strong stability properties. We prove stability of the
canonical decomposition under convergence in quadratic variation, and we
establish stability under transformations. In the latter result,
simultaneous convergence in quadratic variation and in the Skorokhod
\(J_1\)-topology yields convergence of both the martingale and singular
components after transformation. The proof relies on a threshold
isolation principle for jump structures, allowing large jumps to be
separated from small-jump contributions.

\medskip

The article is organized as follows. Section 2 introduces notation and
preliminary material. Section 3 develops the basic theory of
Special--Dirichlet processes and establishes uniqueness of the canonical
decomposition. Section 4 proves stability of the decomposition itself.
Section 5 establishes closure under nonsmooth transformations and proves
stability of the transformed decompositions.

\section{Preliminaries}

We assume that all processes are defined on a common filtered probability
space
\[
(\Omega,\mathcal F,\{\mathcal F_t\}_{t\ge0},\mathbb P),
\]
and that all processes are adapted to a filtration satisfying the usual
conditions.

The term refining sequence will refer to a sequence
\[
D_n=\{\tau_0^n,\tau_1^n,\ldots,\tau_{m_n}^n\}
\]
of finite increasing families of stopping times satisfying
\[
0=\tau_0^n\le\tau_1^n\le\cdots\le\tau_{m_n}^n=t,
\]
and
\[
|D_n|
:=
\max_{0\le i<m_n}
(\tau_{i+1}^n-\tau_i^n)
\xrightarrow{\mathbb P}0.
\]

We say that a c\`adl\`ag process \(X\) admits a quadratic variation
\([X]\) if there exists an increasing continuous process \([X]^c\) such
that
\begin{align}\label{quad}
[X]_s
=
[X]^c_s
+
\sum_{u\le s}(\Delta X_u)^2,
\qquad 0\le s\le t,
\end{align}
and there exists a refining sequence \(\{D_n\}_n\) such that
\[
(S_n(X))_s
:=
\sum_{i=0}^{m_n-1}
\Bigl(
X_{\tau_{i+1}^n\wedge s}
-
X_{\tau_i^n\wedge s}
\Bigr)^2
\]
satisfies
\begin{align}\label{partialsums}
(S_n(X))_s
\xrightarrow{\mathbb P}
[X]_s
\qquad\text{as }n\to\infty,
\end{align}
for every \(0\le s\le t\).

We say that two c\`adl\`ag processes \(X\) and \(Y\) admit a
covariation \([X,Y]\) if there exists a continuous finite variation
process \([X,Y]^c\) such that
\[
[X,Y]_s
=
[X,Y]^c_s
+
\sum_{u\le s}\Delta X_u\,\Delta Y_u,
\qquad 0\le s\le t,
\]
and there exists a refining sequence \(\{D_n\}_n\) such that
\[
S_n(X,Y)_s
:=
\sum_{i=0}^{m_n-1}
\Bigl(
X_{\tau_{i+1}^n\wedge s}
-
X_{\tau_i^n\wedge s}
\Bigr)
\Bigl(
Y_{\tau_{i+1}^n\wedge s}
-
Y_{\tau_i^n\wedge s}
\Bigr)
\]
satisfies
\[
S_n(X,Y)_s
\xrightarrow{\mathbb P}
[X,Y]_s
\qquad\text{as }n\to\infty,
\]
for every \(0\le s\le t\).

By this definition, Dirichlet processes form a subclass of the processes
admitting quadratic variation. Note, however, that the decomposition
\[
X=Z+C
\]
is generally not unique, since any continuous finite variation component
may be transferred between \(Z\) and \(C\).

Given a c\`adl\`ag process \(X\) and a stopping time \(T\), we define
the stopped process
\[
X^T_t:=X_{t\wedge T}.
\]
We also define the supremum process
\[
X_t^*:=\sup_{s\le t}|X_s|.
\]

\begin{defin}\label{loc}
A property of a stochastic process is said to hold locally
(respectively pre-locally) if there exists an increasing sequence of
stopping times \(T_n\uparrow\infty\) such that the property holds for
\(X^{T_n}\) (respectively \(X^{T_n-}\)) for every \(n\).
\end{defin}

The following result is taken from \cite{Collectanea}
\begin{lemma}\label{triangle}
Assume that \(X^1,\ldots,X^n\) admit quadratic variations and
that all covariations needed below exist. Then
\[
        \left[\sum_{k=1}^n X^k\right]_t^{1/2}
        \le
        \sum_{k=1}^n [X^k]_t^{1/2}.
\]
\end{lemma}

We also have the analogous statement for purely discontinuous quadratic variation.

\begin{lemma}
\label{lem:disc-qv-triangle}
Let \(X^1,\ldots,X^n\) be c\`adl\`ag processes such that
\[
        \sum_{s\le t}(\Delta X^k_s)^2<\infty,
        \qquad k=1,\ldots,n.
\]
Then
\[
        \left(
        \sum_{s\le t}
        \left(\sum_{k=1}^n \Delta X^k_s\right)^2
        \right)^{1/2}
        \le
        \sum_{k=1}^n
        \left(
        \sum_{s\le t}(\Delta X^k_s)^2
        \right)^{1/2}.
\]
\end{lemma}

\begin{proof}
This is triangle inequality in \(\ell^2\). Applied to
the sequences
\[
        a^k_s:=\Delta X^k_s,
        \qquad s\le t,\quad k=1,\ldots,n,
\]
it gives
\[
        \left(
        \sum_{s\le t}
        \left(\sum_{k=1}^n a^k_s\right)^2
        \right)^{1/2}
        \le
        \sum_{k=1}^n
        \left(
        \sum_{s\le t}(a^k_s)^2
        \right)^{1/2}.
\]
Substituting \(a^k_s=\Delta X^k_s\) gives the result.
\end{proof}

\section{Special Dirichlet Processes}\label{sec:special}
In this section we introduce the class of Special--Dirichlet
processes and establish the basic properties needed in the sequel.
In particular, we prove uniqueness of the associated canonical
decomposition.

\begin{defin}\label{DefinSpecial}
Let \(X\) be a càdlàg adapted process on \([0,t]\). We say that \(X\) is a
Special--Dirichlet process if it admits a decomposition
\[
        X = M + \Gamma,
\]
where \(M\) is a local martingale with \(M_0=0\), and \(\Gamma\) is an
adapted càdlàg process such that
\[
        [\Gamma]^c_t = 0,
\]
and the jumps of \(\Gamma\) are carried by a thin predictable set and
satisfy
\[
        \Delta \Gamma_s \in \mathcal{F}_{s-}, \qquad s \le t.
\]
\end{defin}
The additional predictable jump structure imposed on \(\Gamma\)
leads to a canonical decomposition.
The remainder of this section is devoted to establishing this fact.

\begin{lemma}[Uniqueness of the local martingale component]
Let \(X\) be a Special--Dirichlet process on \([0,t]\). Suppose that
\[
        X=M+\Gamma
        =
        \widetilde M+\widetilde\Gamma ,
\]
where both decompositions satisfy Definition~\ref{DefinSpecial}, and
where
\[
        M_0=\widetilde M_0=0 .
\]
Then
\[
        M=\widetilde M
\]
up to indistinguishability on \([0,t]\).
\end{lemma}

\begin{proof}
Set
\[
        N:=M-\widetilde M .
\]
Then \(N\) is a local martingale and
\[
        N=\widetilde\Gamma-\Gamma .
\]

Since
\[
        [\Gamma]^c_t=[\widetilde\Gamma]^c_t=0,
\]
we have
\[
        [N]^c_t
        =
        [\widetilde\Gamma-\Gamma]^c_t
        =
        [\widetilde\Gamma]^c_t
        +
        [\Gamma]^c_t
        -
        2[\widetilde\Gamma,\Gamma]^c_t.
\]
Moreover, by the Cauchy--Schwarz inequality for continuous
covariation,
\[
        \bigl|[\widetilde\Gamma,\Gamma]^c_t\bigr|
        \le
        [\widetilde\Gamma]^c_t{}^{1/2}
        [\Gamma]^c_t{}^{1/2}
        =
        0.
\]
Hence
$
        [\widetilde\Gamma,\Gamma]^c_t=0,
$
and therefore
$
        [N]^c_t=0.
$

It remains to show that \(N\) has no jumps. By the definition of a
Special--Dirichlet process, the jumps of \(\Gamma\) and
\(\widetilde\Gamma\) are contained in thin predictable sets, and their jumps are \(\mathcal F_{s-}\)-measurable. Hence the jump times of
$
        \widetilde\Gamma-\Gamma
$
are carried by a thin predictable set and have \(\mathcal F_{s-}\)-
measurable jumps.

Let \((T_k)_{k\ge1}\) be predictable stopping times exhausting this thin
predictable set. Then, for each \(k\),
\[
        \Delta N_{T_k}
        =
        \Delta\widetilde\Gamma_{T_k}
        -
        \Delta\Gamma_{T_k}
        \in\mathcal F_{T_k-}.
\]
After localization, the jump is integrable. Since \(N\) is a local
martingale and \(T_k\) is predictable,
\[
        \mathbb E[\Delta N_{T_k}\mid\mathcal F_{T_k-}]=0.
\]
Because \(\Delta N_{T_k}\) is already \(\mathcal F_{T_k-}\)-measurable,
we obtain
\[
        \Delta N_{T_k}=0
        \qquad\text{a.s.}
\]
for every \(k\). Thus \(N\) has no jumps.

Consequently,
\[
        [N]_t=[N]^c_t+\sum_{s\le t}(\Delta N_s)^2=0.
\]
Since \(N_0=0\), it follows that
$
        N\equiv0
$
on \([0,t]\), up to indistinguishability. Hence
$
        M=\widetilde M .
$
\end{proof}

The above Lemma motivates the use of the following notation. For a special Dirichlet process $X$ let $\mathfrak{M}(X)$ denote its unique local martingale component and let $\mathfrak{\Gamma}(X)$ denote its unique singular component.
\section{Stability of the Canonical Decomposition}
The uniqueness result of the previous section identifies a canonical
martingale component $\mathfrak{M}(X)$ and a canonical singular component
$\mathfrak{\Gamma}(X)$ for every Special--Dirichlet process \(X\).
A natural question is whether this decomposition is stable under
perturbations of the underlying process.
The purpose of this section is to show that convergence in quadratic
variation propagates to the canonical decomposition itself.
\begin{lemma}[Stability of the Special--Dirichlet decomposition]
\label{lem:SD-decomposition-stability}
Let \(X^n\) and \(X\) be Special--Dirichlet processes on \([0,t]\), with
canonical decompositions
\[
        X^n=M^n+\Gamma^n,
        \qquad
        X=M+\Gamma,
\]
where \(M^n_0=M_0=0\). Assume that
\[
        [X^n-X]_t \xrightarrow{\mathbb P} 0 .
\]
Then
\[
        [\Gamma^n-\Gamma]_t \xrightarrow{\mathbb P} 0
\]
and
\[
        [M^n-M]_t \xrightarrow{\mathbb P} 0 .
\]
\end{lemma}

\begin{proof}
Set
\[
        Y^n:=X^n-X,\qquad
        N^n:=M^n-M,\qquad
        B^n:=\Gamma^n-\Gamma .
\]
Then
\[
        Y^n=N^n+B^n .
\]
By the definition of Special--Dirichlet processes, \(B^n\) has zero
continuous quadratic variation and its jump times are carried by a thin
predictable set. Moreover,
\[
        \Delta B^n_s\in\mathcal F_{s-}
\]
at every jump time of \(B^n\).

Let \((T^n_k)_{k\ge1}\) be predictable stopping times exhausting the jump
times of \(B^n\). Since
\[
        [B^n]^c_t=0,
\]
we have
\[
        [B^n]_t
        =
        \sum_{k\ge1:T^n_k\le t}
        \bigl(\Delta B^n_{T^n_k}\bigr)^2 .
\]

Fix \(\eta>0\), and define the stopping time
\[
        \tau^n_\eta
        :=
        \inf\{s\le t:[Y^n]_s>\eta\}.
\]
On the event \(\{\tau^n_\eta>t\}\), we have \([Y^n]_t\le \eta\).

We first estimate the predictable jump part before \(\tau^n_\eta\). After
standard localization, we may assume that the local martingale \(N^n\) is
square-integrable up to \(t\). For every \(k\), on the event
\(\{T^n_k<\tau^n_\eta\}\), the variable
\[
        \Delta B^n_{T^n_k}
\]
is \(\mathcal F_{T^n_k-}\)-measurable. Since \(N^n\) is a local martingale,
\[
        \mathbb E\!\left[
        \Delta N^n_{T^n_k}
        \mid
        \mathcal F_{T^n_k-}
        \right]
        =0 .
\]
Hence
\[
\begin{aligned}
\mathbb E\!\left[
        \bigl(\Delta Y^n_{T^n_k}\bigr)^2
        \mid
        \mathcal F_{T^n_k-}
        \right]
&=
\mathbb E\!\left[
        \bigl(\Delta N^n_{T^n_k}
        +
        \Delta B^n_{T^n_k}\bigr)^2
        \mid
        \mathcal F_{T^n_k-}
        \right]
\\
&=
        \bigl(\Delta B^n_{T^n_k}\bigr)^2
        +
        \mathbb E\!\left[
        \bigl(\Delta N^n_{T^n_k}\bigr)^2
        \mid
        \mathcal F_{T^n_k-}
        \right]
\\
&\ge
        \bigl(\Delta B^n_{T^n_k}\bigr)^2 .
\end{aligned}
\]
Therefore,
\[
        \mathbb E\!\left[
        \bigl(\Delta B^n_{T^n_k}\bigr)^2
        \mathbf 1_{\{T^n_k<\tau^n_\eta\}}
        \right]
        \le
        \mathbb E\!\left[
        \bigl(\Delta Y^n_{T^n_k}\bigr)^2
        \mathbf 1_{\{T^n_k<\tau^n_\eta\}}
        \right].
\]
Summing over \(k\) and using monotone convergence gives
\[
\begin{aligned}
\mathbb E\!\left[
        \sum_{k\ge1:T^n_k<\tau^n_\eta}
        \bigl(\Delta B^n_{T^n_k}\bigr)^2
        \right]
&\le
\mathbb E\!\left[
        \sum_{k\ge1:T^n_k<\tau^n_\eta}
        \bigl(\Delta Y^n_{T^n_k}\bigr)^2
        \right]
\\
&\le
\mathbb E\!\left[
        \sum_{s<\tau^n_\eta}
        \bigl(\Delta Y^n_s\bigr)^2
        \right]
\\
&\le
\eta .
\end{aligned}
\]
The last inequality follows from the definition of \(\tau^n_\eta\).

By Markov's inequality,
\[
\mathbb P\!\left(
        \sum_{k\ge1:T^n_k<\tau^n_\eta}
        \bigl(\Delta B^n_{T^n_k}\bigr)^2
        >\varepsilon
        \right)
        \le
        \frac{\eta}{\varepsilon}.
\]
Hence
\[
\begin{aligned}
\mathbb P\!\left([B^n]_t>\varepsilon\right)
&\le
\mathbb P(\tau^n_\eta\le t)
+
\mathbb P\!\left(
        \sum_{k\ge1:T^n_k<\tau^n_\eta}
        \bigl(\Delta B^n_{T^n_k}\bigr)^2
        >\varepsilon
        \right)
\\
&\le
\mathbb P([Y^n]_t>\eta)
+
\frac{\eta}{\varepsilon}.
\end{aligned}
\]
Since \([Y^n]_t=[X^n-X]_t\to0\) in probability, we obtain
\[
        \limsup_{n\to\infty}
        \mathbb P\!\left([B^n]_t>\varepsilon\right)
        \le
        \frac{\eta}{\varepsilon}.
\]
Letting \(\eta\downarrow0\), we conclude that
\[
        [B^n]_t=[\Gamma^n-\Gamma]_t
        \xrightarrow{\mathbb P}0 .
\]

It remains to prove the convergence of the martingale parts. Since
\[
        N^n=Y^n-B^n,
\]
we have, by the triangle inequality for quadratic variation,
\[
        [N^n]_t^{1/2}
        \le
        [Y^n]_t^{1/2}
        +
        [B^n]_t^{1/2}.
\]
Therefore
\[
        [M^n-M]_t=[N^n]_t
        \xrightarrow{\mathbb P}0 .
\]
\end{proof}



\section{Transformations of Special–Dirichlet Processes}
Having established stability of the canonical decomposition, we now
turn to transformation theory.
The main question is whether the Special--Dirichlet structure is
preserved under nonsmooth changes of variables.
Our first result shows that the class is closed under a broad family
of transformations and provides an explicit decomposition of the
transformed process. 

For a function \(f:\mathbb R\to\mathbb R\) we define
\[
f'(x)
=
\limsup_{h\downarrow0}
\frac{f(x+h)-f(x)}{h}.
\]
Throughout the paper, \(f'\) refers to this upper Dini derivative.

\begin{thm}[Closure under nonsmooth transformations]
\label{thm:SD-closure}
Let \(X\) be a Special--Dirichlet process on \([0,t]\), with canonical
decomposition
\[
        X=M+\Gamma,
        \qquad M_0=0,
\]
where \(M\) is a local martingale and \(\Gamma\) is an adapted càdlàg
process such that
$
        [\Gamma]^c_t=0,
$
and whose jumps are carried by a thin predictable set and satisfy
$
        \Delta\Gamma_s\in\mathcal F_{s-}.
$
Let \(f\) be the primitive of a locally bounded function \(f'\) with an
at most countable set of discontinuities. Assume that
\[
        \int_0^t
        \mathbf 1_{\{X_s\notin\operatorname{diff}(f)\}}\,d[X]^c_s=0.
\]
Then \(f(X)\) is again a Special--Dirichlet process.

More precisely, 
\[
        f(X)=M^f+\Gamma^f,
\]
where
\[
        M^f_t
        :=
        \int_0^t f'(X_{s-}+\Delta\Gamma_s)\,dM^c_s
        +
        \int_0^t\int_{\mathbb R}
        \left(f(X_{s-}+\Delta\Gamma_s+x)-f(X_{s-}+\Delta\Gamma_s)\right)(\mu^M-\nu^M)(ds,dx)
\]
is a local martingale, and \(\Gamma^f:=f(X)-M^f\) satisfies
$
        [\Gamma^f]^c_t=0.
$

Moreover,
\[
        \Delta\Gamma^f_s
        =
        f(X_{s-}+\Delta\Gamma_s)-f(X_{s-})
        +
        \int_{\mathbb R}
        \left(f(X_{s-}+\Delta\Gamma_s+x)-f(X_{s-}+\Delta\Gamma_s)\right)\nu^M(\{s\},dx).
\]
In particular, the jumps of \(\Gamma^f\) are carried by a thin predictable
set and satisfy
\[
        \Delta\Gamma^f_s\in\mathcal F_{s-}.
\]
\end{thm}

\begin{proof}
Write
\[
        \widehat X_s:=X_{s-}+\Delta\Gamma_s.
\]
Since \(X_{-}\) is predictable and the jump times of \(\Gamma\) are carried by
a thin predictable set with \(\mathcal F_{s-}\)-measurable jumps,
\(\widehat X\) is predictable. Let
\[
        \Phi(s,x):=f(\widehat X_s+x)-f(\widehat X_s).
\]
After localization, we may assume that \(X\), \(M\), and \(\Gamma\) are
bounded on \([0,t]\), and that \(f'\) is bounded on the relevant compact
interval. Therefore, on the relevant compact interval, there exists
\(C>0\) such that, whenever both
\[
        \widehat X_s
        \quad\text{and}\quad
        \widehat X_s+x
\]
belong to this interval,
\[
        |\Phi(s,x)|
        \le C|x|.
\] 
Since \(M^c\) is a
continuous local martingale,
\[
        \int_0^\cdot f'(\widehat X_s)\,dM^c_s
\]
is a local martingale, and
\[
        \int_0^\cdot\int_{\mathbb R}
        \Phi(s,x)(\mu^M-\nu^M)(ds,dx)
\]
is a well-defined purely discontinuous local martingale. Hence \(M^f\)
is a local martingale.

Set
\[
        \Gamma^f:=f(X)-M^f.
\]

We first compute the jumps of \(\Gamma^f\). Put
\[
        u_s:=\Delta M_s,
        \qquad
        v_s:=\Delta\Gamma_s,
        \qquad
        y_s:=X_{s-}.
\]
Then
\[
        \widehat X_s=y_s+v_s,
        \qquad
        X_s=\widehat X_s+u_s.
\]
Thus
\[
        \Delta f(X)_s
        =
        f(\widehat X_s+u_s)-f(y_s).
\]
On the other hand, since the first integral in \(M^f\) is continuous,
\[
\begin{aligned}
\Delta M^f_s
&=
        \Phi(s,u_s)
        -
        \int_{\mathbb R}\Phi(s,x)\nu^M(\{s\},dx)
\\
&=
        f(\widehat X_s+u_s)-f(\widehat X_s)
        -
        \int_{\mathbb R}\Phi(s,x)\nu^M(\{s\},dx).
\end{aligned}
\]
Therefore
\[
\begin{aligned}
\Delta\Gamma^f_s
&=
        \Delta f(X)_s-\Delta M^f_s
\\
&=
        f(\widehat X_s)-f(y_s)
        +
        \int_{\mathbb R}\Phi(s,x)\nu^M(\{s\},dx)
\\
&=
        f(X_{s-}+\Delta\Gamma_s)-f(X_{s-})
        +
        \int_{\mathbb R}
        \Phi(s,x)\nu^M(\{s\},dx).
\end{aligned}
\]
The first term is \(\mathcal F_{s-}\)-measurable by the defining jump
condition on \(\Gamma\). The second term is \(\mathcal F_{s-}\)-measurable
because \(\nu^M\) is predictable and the integrand is predictable. Hence
\[
        \Delta\Gamma^f_s\in\mathcal F_{s-}.
\]
The jumps of \(\Gamma^f\) are supported on the union of the jump times of
\(\Gamma\) and the predictable atoms of \(\nu^M\), hence on a thin
predictable set.

It remains to prove that
\[
        [\Gamma^f]^c_t=0.
\]

Since \(f'\) is a locally bounded Baire class \(1\) function, there
exist polynomials \(p_n\) such that
\[
        p_n(x)\to f'(x)
\]
for every \(x\in[-m,m]\).
Moreover, since \(f'\) is bounded on \([-m,m]\), the approximation may
be chosen so that
\[
        \sup_n\sup_{|x|\le m}|p_n(x)|<\infty.
\]
Define
\[
        f_n(x):=f(-m)+\int_{-m}^x p_n(u)\,du,
\]
and
\[
        \Phi_n(s,x):=f_n(\widehat X_s+x)-f_n(\widehat X_s).
\]

Let \(M^n\) be defined as \(M^f\) with \(f',\Phi\) replaced by
\(p_n,\Phi_n\), namely
\[
        M^n_t
        :=
        \int_0^t p_n(\widehat X_s)\,dM^c_s
        +
        \int_0^t\int_{\mathbb R}
        \Phi_n(s,x)(\mu^M-\nu^M)(ds,dx),
\]
and set
\[
        \Gamma^n:=f_n(X)-M^n.
\]

\medskip

\textit{Step 1: smooth case.}
Since \(f_n\in C^1\), the standard change-of-variables formula applied to
\(X=M+\Gamma\), together with \([\Gamma]^c=0\), yields a decomposition
in which the residual satisfies
$
        [\Gamma^n]^c_t=0.
$

\medskip

\textit{Step 2: decomposition of the difference.}
Write
\[
        \Gamma^f
        =
        \bigl(f(X)-f_n(X)\bigr)
        -
        (M^f-M^n)
        +
        \Gamma^n.
\]

By the triangle inequality for continuous quadratic variation,
\[
\begin{aligned}
[\Gamma^f]_t^{c\,1/2}
&\le
[f(X)-f_n(X)]_t^{c\,1/2}
+
[M^f-M^n]_t^{c\,1/2}
+
[\Gamma^n]_t^{c\,1/2}.
\end{aligned}
\]
Since \([\Gamma^n]^c_t=0\), it suffices to show that the first two terms
converge to zero.

\medskip

\textit{Step 3: control of \(f(X)-f_n(X)\).}
Set \(h_n:=f-f_n\). Then \(h_n\) is the primitive of \(f'-p_n\).
By Lemma~6.2 in \cite{Collectanea} and the non-charging assumption,
\[
[h_n(X)]^c_t
=
\int_0^t
\bigl(f'(X_s)-p_n(X_s)\bigr)^2\,d[X]^c_s.
\]

After localization, assume \(|X_s|\le m\) and
\[
|f'(x)|\le K,
\qquad
|p_n(x)|\le K
\quad (|x|\le m).
\]
Then
\[
|f'(X_s)-p_n(X_s)|^2 \le 4K^2,
\]
and since \(p_n(x)\to f'(x)\) for every \(x\in[-m,m]\), we have, for
every \(\omega\) and every \(s\le t\) on the localized set,
\[
        (f'(X_s(\omega))-p_n(X_s(\omega)))^2\to0.
\]
Since \(d[X]^c_s\) is a finite measure on \([0,t]\), dominated convergence
gives
\[
        [f(X)-f_n(X)]^c_t=[h_n(X)]^c_t\longrightarrow0.
\]

\medskip

\textit{Step 4: control of the martingale difference.}
We have
\[
M^f_t-M^n_t
=
\int_0^t (f'(\widehat X_s)-p_n(\widehat X_s))\,dM^c_s
+
\int_0^t\int_{\mathbb R}
(\Phi-\Phi_n)(s,x)(\mu^M-\nu^M)(ds,dx).
\]
The second term is purely discontinuous, hence
\[
[M^f-M^n]^c_t
=
\int_0^t
(f'(\widehat X_s)-p_n(\widehat X_s))^2\,d[M]^c_s.
\]

Since \([M]^c\) is continuous, the measure \(d[M]^c_s\) has no atoms.
The jump set of the càdlàg process \(\Gamma\) is at most countable on
\([0,t]\). Hence \(d[M]^c_s\) does not charge the set
\(\{s:\Delta\Gamma_s\ne0\}\). Consequently,
\[
        \widehat X_s=X_{s-}+\Delta\Gamma_s=X_{s-}
        \quad d[M]^c_s\text{-a.e.}
\]
Moreover, \(d[M]^c_s=d[X]^c_s\), since \([\Gamma]^c=0\) and the
continuous quadratic variation of \(X\) is carried by the continuous
martingale part of \(M\). Therefore the same dominated convergence
argument gives
\[
        [M^f-M^n]^c_t \longrightarrow 0.
\]

\medskip

\textit{Step 5: conclusion.}
Combining the estimates,
$
[\Gamma^f]_t^{c\,1/2}\to0,
$
hence
$
        [\Gamma^f]^c_t=0.
$
\end{proof}

The preceding theorem applies to general Special--Dirichlet
processes. For special semimartingales, Lowther's non-charging
condition is automatically satisfied, yielding the following
immediate consequence.

\begin{corollary}
Let \(Z\) be a special semimartingale on \([0,t]\), and let \(f \in C^1(\mathbb{R})\).
Then \(f(Z)\) is a special Dirichlet process.
\end{corollary}
\begin{proof}
This follows from the above theorem and the fact that 
\[
        \int_0^t
        \mathbf 1_{\{X_s\notin\operatorname{diff}(f)\}}\,d[X]^c_s=0
\]
is true whenever $X$ is a semimartingale (as was shown in \cite{Low}).
\end{proof}

\subsection{Stability under Transformations}
We now combine the structural stability result of Section 4 with the
transformation theorem above. The goal is to show that convergence of
the underlying processes propagates through the transformation and
remains visible at the level of the canonical decomposition.
As mentioned earlier, for a special Dirichlet process X let $\mathfrak{M}(X)$ denote its unique local martingale component and let $\mathfrak{\Gamma}(X)$ denote its unique predictable component. We shall need two lemmas to prove the theorem in this section. The first lemma shows that convergence in the \(J_1\)-topology,
combined with vanishing quadratic variation of the error process,
implies convergence of the corresponding left limits.

\begin{lemma}\label{lem:QVJ1}
Let \(X^n\) and \(X\) be real-valued cadlag functions on \([0,t]\).
Assume that
\[
        X^n \to X
        \quad\text{in the } J_1\text{-topology}
\]
and that
\[
        [X^n-X]_t \to 0 .
\]
Then, for every \(s\in(0,t]\),
\[
        X^n_{s-}\to X_{s-}.
\]
\end{lemma}

\begin{proof}
Fix \(s\in(0,t]\), and write
\[
        L:=X_{s-},\qquad R:=X_s .
\]
We use the following standard consequence of \(J_1\)-convergence: if
\(x_n\to x\) in the \(J_1\)-topology and \(u_n\to s\), then every cluster
point of \(x_n(u_n)\) belongs to the completed graph segment
\[
        [[x(s-),x(s)]]
        :=
        \{\theta x(s-)+(1-\theta)x(s):\theta\in[0,1]\}.
\]

Put
\[
        Y^n:=X^n-X .
\]
Since the quadratic variation contains the squared jumps,
\[
        \bigl(\Delta Y^n_s\bigr)^2\le [Y^n]_t .
\]
Hence
$
        \Delta Y^n_s\to0 .
$
Equivalently,
\[
        \Delta X^n_s-\Delta X_s\to0,
\]
and therefore
\[
        X^n_s-X^n_{s-}\to R-L .
\]

Let \(a\) be an arbitrary cluster point of \(X^n_{s-}\). Passing to a
subsequence, we may assume that
$
        X^n_{s-}\to a .
$
For each \(n\), choose
\[
        u_n\in (s-1/n,s)\cap[0,t]
\]
such that
\[
        |X^n(u_n)-X^n_{s-}|<1/n .
\]
This is possible by the existence of the left limit \(X^n_{s-}\).
Then \(u_n\to s\), and along the chosen subsequence
\[
        X^n(u_n)\to a .
\]
By the completed graph consequence of \(J_1\)-convergence,
$
        a\in [[L,R]] .
$

Passing to a further subsequence if necessary, assume also that
$
        X^n_s\to b .
$
Again by the completed graph consequence, now with \(u_n=s\), we have
$
        b\in [[L,R]] .
$
On the other hand,
\[
        b-a
        =
        \lim_{n\to\infty} (X^n_s-X^n_{s-})
        =
        R-L .
\]
Since \(a,b\in [[L,R]]\) and \(b-a=R-L\), it follows in the real-valued
case that
\[
        a=L,\qquad b=R .
\]
Thus every cluster point of \(X^n_{s-}\) equals \(X_{s-}\). Consequently,
$
        X^n_{s-}\to X_{s-}.
$
\end{proof}
The second lemma, taken from \cite{CompJump}, provides a threshold
isolation principle: for arbitrarily small jump-size thresholds, one can
find a neighborhood around the threshold which, with probability tending
to one, eventually contains no jumps of either \(X\) or \(X^n\).
\begin{lemma}[Threshold Isolation Principle]
\label{lem:thres}
Let $X$ and $\{X^n\}_{n\ge1}$ be c\`adl\`ag processes on $[0,t]$ such that
\[
[X^n-X]_t \to 0
\qquad \text{in probability}.
\]
Fix $r>0$, and for each $k\ge1$ define
\[
\delta(k,r):=\frac{r}{2^{k+2}},
\qquad
L(r,k):=r\Bigl(1-\frac{3}{2^{k+2}}\Bigr),
\]
and
\[
A_k(r)
:=
\Bigl\{
\omega:
\bigl||\Delta X_s(\omega)|-L(r,k)\bigr|>\delta(k,r)
\text{ for all } s\le t
\Bigr\}.
\]
Then
\[
\mathbb P\Bigl(\bigcup_{k\ge1}A_k(r)\Bigr)=1.
\]

Moreover, for each fixed $k\ge1$, if we define
\[
B_n(k,r)
:=
\Bigl\{
\sup_{s\le t}|\Delta(X^n-X)_s|<\delta(k,r)
\Bigr\},
\]
then
$
\lim_{n\to\infty}\mathbb P\bigl(B_n(k,r)\bigr)=1.
$
On the event $A_k(r)\cap B_n(k,r)$, the classification of jumps relative to the threshold
$L(r,k)$ is preserved, i.e.
\[
|\Delta X_s|\le L(r,k)-\delta(k,r)
\quad\Longrightarrow\quad
|\Delta X^n_s|<L(r,k),
\]
and
\[
|\Delta X_s|\ge L(r,k)+\delta(k,r)
\quad\Longrightarrow\quad
|\Delta X^n_s|>L(r,k),
\]
for all $s\le t$.

In particular, on $A_k(r)\cap B_n(k,r)$,
\[
1_{\{|\Delta X^n_s|\le L(r,k)\}}
\le
1_{\{|\Delta X_s|\le L(r,k)\}}
\qquad \text{for all } s\le t,
\]
and therefore
\[
\sum_{s\le t} (\Delta X^n_s)^2 1_{\{|\Delta X^n_s|\le L(r,k)\}}
\le
2[X^n-X]_t
+
2\sum_{s\le t} (\Delta X_s)^2 1_{\{|\Delta X_s|\le L(r,k)\}}.
\]
\end{lemma}

With these tools we can now prove the following.
\begin{thm}\label{C1}
Let $f\in C^1$. Let $\{X^n\}_n$ be special Dirichlet processes such that for each $n$, $X^n$ and $X$ have quadratic variations along the same refining sequence, that $X^n\xrightarrow{J_1}X$ in probability, $[X^n-X]_t\xrightarrow{\P}0$. Then
\\$[\mathfrak{\Gamma}(f(X^n))-\mathfrak{\Gamma}(f(X))]_t \xrightarrow{\P} 0$ and $\mathfrak{M}(f(X^n))\xrightarrow{ucp}\mathrm{M}(f(X))$ as $n\to \infty$.
\end{thm}
\begin{proof}
We first establish that $[f(X^n)-f(X)]_t\xrightarrow{\P}0$. According to Lemma \ref{lem:SD-decomposition-stability} it then follows that $[\mathfrak{M}(f(X^n))-\mathfrak{M}(f(X))]_t\xrightarrow{\P}0$ and $[\mathfrak{\Gamma}(f(X^n))-\mathfrak{\Gamma}(f(X))]_t\xrightarrow{\P}0$. By localization and the Burkholder-David-Gundy inequality it then follows that $\mathfrak{M}(f(X^n))\xrightarrow{ucp}\mathfrak{M}(f(X))$. 

Let $\{n_j\}_{j\in\N}$ be an arbitrary subsequence and extract a further subsequence $\{n_{j_l}\}_{l\in\N}$ such that both
$
[X^{n_{j_l}}-X]_t\to 0
$ a.s.
and
$
X^{n_{j_l}}\xrightarrow{J_1}X
$ a.s.. according to Lemma \ref{lem:QVJ1} we conclude that
\[
X^{n_{j_l}}_{s-}\xrightarrow{a.s.}X_{s-}
\qquad\text{for every } s\in[0,t].
\]
Since the sequence $\{n_j\}_{j\in\N}$ is arbitrary, it suffices to show that $[f(X^{n_{j_l}})-f(X)]_t\xrightarrow{\P}0$ as $l\to \infty$, so by relabelin we may assume that $X^{n}_{s-}\xrightarrow{a.s.}X_{s-}$. If we let 
$$B_R=\left\{\sup_n(X^n)^*_t\vee X^*_t\le R\right\}$$ 
then 
$$\P\left(\cup_{R} B_R\right)=\lim_{R\to\infty}\P(B_R)=1,$$ 
so we may assume that $X,\Delta X$, $\{X^n\}_n$ and $\{\Delta X^n\}_n$ are all uniformly bounded by the constant $R$.  We will let $L=\sup_{x\in[-R-a,R+a]} |f'(x)|$ (which must be finite as cadlag functions are bounded on compacts).
According to Theorem 2.1 in \cite{Low}	
\begin{align*}
f(X_s)=\int_0^sf'(X_{u-})dM_u+V_s
\end{align*}
and similarly
\begin{align*}
f(X^n_s)=\int_0^sf'(X^n_{u-})dM^n_u+V^n_s
\end{align*}

 Let $X^n=M^n+\Gamma^n$ and $X=M+\Gamma$ be the canonical decompositions of $X$ and $X^n$ respectively. 
 Using Lemma 2.5 in \cite{Collectanea},
\begin{align}\label{fuljavel}
[f(X^n)-f(X)]_t^{\frac12}
&\le
\left[\int_0^. f'(X_{s-})dM_s-\int_0^. f'(X^n_{s-})dM^n_s\right]_t^{\frac12}
+
\left[V^n-V\right]_t^{\frac12}
\end{align}
for the first term on the right-hand side of \eqref{fuljavel} we have, due to Lemma \ref{triangle}
\begin{align}\label{Mterms}
\left[\int_0^. f'(X_{s-})dM_s-\int_0^. f'(X^n_{s-})dM^n_s\right]_t^{\frac12}
&\le
\left[\int_0^. \left(f'(X_{s-})- f'(X^n_{s-})\right) dM_s\right]_t^{\frac12}
+
\left[\int_0^. \left(f'(X^n_{s-})\right) d(M_s-M^n_s)\right]_t^{\frac12}\nonumber
\\
&=
\left( \int_0^t \left(f'(X_{s-})- f'(X^n_{s-})\right)^2 d[M]_s\right)^{\frac12}
+
\left(\int_0^t f'(X^n_{s-})^2 d[M-M^n]_s\right)_t^{\frac12}.
\end{align}
On the set $B_R$ we have that $\left(f'(X_{s-})- f'(X^n_{s-})\right)^2\le 4L^2$ so it follows by dominated convergence (applied pathwise) that
$$\lim_{n\to\infty}\left( \int_0^t \left(f'(X_{s-})- f'(X^n_{s-})\right)^2 d[M]_s\right)^{\frac12}1_{B_R}=0. $$
Hence, for every $\epsilon>0$
$$\P\left( \left( \int_0^t \left(f'(X_{s-})- f'(X^n_{s-})\right)^2 d[M]_s\right)^{\frac12}\ge \epsilon\right)
\le
 \P\left(\left( \int_0^t \left(f'(X_{s-})- f'(X^n_{s-})\right)^2 d[M]_s\right)^{\frac12}1_{B_R}\ge \epsilon\right) + \P(B_R^c)$$
where the first term converges to zero as $n\to\infty$ and the second term as $R\to \infty$. For the second term on the right-hand side of \eqref{Mterms}, note that $[M-M^n]_t\xrightarrow{\P}0$ by Lemma \ref{lem:SD-decomposition-stability} and therefore
$$\P\left( \left(\int_0^t f'(X^n_{s-})^2 d[M-M^n]_s\right)_t^{\frac12}\ge \epsilon\right)
\le
\P\left( L[M-M^n]_t^{\frac12}\ge \epsilon\right) + \P(B_R^c)$$
which converges to zero. This takes care of the first term on the right-hand side of \eqref{fuljavel}. For the second term on the right-hand side of note that, by Lemma \ref{lem:disc-qv-triangle}
\begin{align*}
[V^n-V]^{\frac12}_t=\left([V^n-V]^d_t\right)^{\frac12}
\le
\left[\int_0^. f'(X_{s-})dM_s-\int_0^. f'(X^n_{s-})dM^n_s\right]_t^{\frac12}
+
\left([f(X^n)-f(X)]^d_t\right)^{\frac12}.
\end{align*}
We have already shown that the first term on the right-hand side above converges to zero so to finish the proof it only remains to prove that also the second term vanishes. Set
\[
        F^n:=f(X^n)-f(X),
        \qquad
        D^n:=X^n-X .
\]
We work on the localized set \(B_R\). On this set \(f\) is Lipschitz on
the relevant compact interval; let \(K_R\) denote a corresponding
Lipschitz constant. Thus, whenever both arguments stay in the localized
range,
\[
        |f(x)-f(y)|\le K_R |x-y|.
\]

Fix \(r>0\). By Lemma~\ref{lem:thres}, choose \(k\ge1\) such that the
event \(A_k(r)\) has probability arbitrarily close to one, and put
\[
        \rho:=L(r,k).
\]
On \(A_k(r)\cap B_n(k,r)\), the classification of jumps above and below
the threshold \(\rho\) is preserved. In particular,
\[
        |\Delta X^n_s|>\rho
        \quad\Longleftrightarrow\quad
        |\Delta X_s|>\rho .
\]
Let
\[
        J_\rho:=\{s\le t:|\Delta X_s|>\rho\}.
\]
Since \(X\) has finite quadratic variation, \(J_\rho\) is finite.

We split
\[
        [F^n]^d_t
        =
        \sum_{s\in J_\rho}(\Delta F^n_s)^2
        +
        \sum_{s\le t:\,|\Delta X_s|\le \rho}(\Delta F^n_s)^2
\]
on the event \(A_k(r)\cap B_n(k,r)\).

First consider the large-jump part. If \(s\in J_\rho\), then by
Lemma~\ref{lem:QVJ1},
\[
        X^n_{s-}\to X_{s-}.
\]
Moreover,
\[
        \Delta X^n_s-\Delta X_s
        =
        \Delta D^n_s\to0,
\]
because
\[
        |\Delta D^n_s|^2\le [D^n]_t\to0.
\]
Hence also \(X^n_s\to X_s\). Since \(f\) is continuous,
\[
        \Delta F^n_s
        =
        \Delta f(X^n)_s-\Delta f(X)_s
        \longrightarrow 0
\]
for every \(s\in J_\rho\). Since \(J_\rho\) is finite,
\[
        \sum_{s\in J_\rho}(\Delta F^n_s)^2\longrightarrow 0.
\]

It remains to control the small-jump part. On
\(A_k(r)\cap B_n(k,r)\), if \(|\Delta X_s|\le \rho\), then
\(|\Delta X^n_s|\le \rho\). Therefore, using the Lipschitz bound,
\[
\begin{aligned}
|\Delta F^n_s|
&=
\left|
        \bigl(f(X^n_s)-f(X^n_{s-})\bigr)
        -
        \bigl(f(X_s)-f(X_{s-})\bigr)
\right|
\\
&\le
        K_R |\Delta X^n_s|
        +
        K_R |\Delta X_s|.
\end{aligned}
\]
Consequently,
\[
        (\Delta F^n_s)^2
        \le
        2K_R^2\bigl((\Delta X^n_s)^2+(\Delta X_s)^2\bigr).
\]
Summing over the small jumps gives
\[
\begin{aligned}
\sum_{s\le t:\,|\Delta X_s|\le \rho}(\Delta F^n_s)^2
&\le
        2K_R^2
        \sum_{s\le t:\,|\Delta X_s|\le \rho}(\Delta X^n_s)^2
        +
        2K_R^2
        \sum_{s\le t:\,|\Delta X_s|\le \rho}(\Delta X_s)^2.
\end{aligned}
\]
Since
\[
        \Delta X^n_s=\Delta X_s+\Delta D^n_s,
\]
we have
\[
        (\Delta X^n_s)^2
        \le
        2(\Delta X_s)^2+2(\Delta D^n_s)^2.
\]
Thus
\[
\begin{aligned}
\sum_{s\le t:\,|\Delta X_s|\le \rho}(\Delta F^n_s)^2
&\le
        6K_R^2
        \sum_{s\le t:\,|\Delta X_s|\le \rho}(\Delta X_s)^2
        +
        4K_R^2
        \sum_{s\le t}(\Delta D^n_s)^2
\\
&\le
        6K_R^2
        \sum_{s\le t:\,|\Delta X_s|\le \rho}(\Delta X_s)^2
        +
        4K_R^2 [D^n]_t .
\end{aligned}
\]
Letting \(n\to\infty\), the second term vanishes. Hence, on
\(A_k(r)\cap B_n(k,r)\),
\[
        \limsup_{n\to\infty}[F^n]^d_t
        \le
        6K_R^2
        \sum_{s\le t:\,|\Delta X_s|\le \rho}(\Delta X_s)^2.
\]
Finally, as \(\rho\downarrow0\),
\[
        \sum_{s\le t:\,|\Delta X_s|\le \rho}(\Delta X_s)^2
        \longrightarrow0,
\]
because
$
        \sum_{s\le t}(\Delta X_s)^2<\infty.
$

Since \(r>0\) was arbitrary and the events \(A_k(r)\cap B_n(k,r)\) have
probability tending to one, it follows that
\[
        [f(X^n)-f(X)]^d_t
        =
        [F^n]^d_t
        \xrightarrow{\mathbb P}0.
\]

Combining this with the previous estimates yields
\[
        [f(X^n)-f(X)]_t\xrightarrow{\mathbb P}0.
\]
This completes the proof.
\end{proof}

\hfill


\begin{thebibliography}{99}

\bibitem[Kennerberg and Wiktorsson (2026)]{Collectanea}
Kennerberg, P. and Wiktorsson, M. (2026).
Stability in quadratic variation.
Collectanea Mathematica.
\bibitem[Coquet, M\'emin and S{\l}omi{\'n}ski(2003)]{NonCont}
Coquet, F., M\'emin, J. and S{\l}omi{\'n}ski, L. (2003).
On Non-Continuous Dirichlet Processes.
\textit{Journal of Theoretical Probability}, {\bf 16}, 197.
\bibitem[F\"ollmer(1981)]{Fol}
F\"ollmer, H. (1981). Dirichlet processes. In \textit{Lecture Notes in Maths.}, Vol. {\bf 851}, pp. 476-–478, Springer-Verlag, Berlin/Heidelberg/New York.
\bibitem[F\"ollmer(1981)]{FolIto}
  F\"ollmer, H. (1981). Calcul d'Ito sans probabilit\'es. In \textit{S\'eminaire de probabilit\'es (Strasbourg).}, Vol. {\bf 851}, , pp. 143-150, Springer-Verlag, Berlin/Heidelberg/New York.
\bibitem[Jacod and Shiryaev(2003)]{JACOD}
Jacod, J. and Shiryaev, A. N. (2003). 
\textit{Limit Theorems for Stochastic Processes}. Springer-Verlag.
\bibitem[Liptser and Shiryaev(1986)]{Mart}
Liptser, R. and Shiryaev, A. (1986) \textit{Theory of Martingales}. Springer-Verlag.
\bibitem[Lowther(2010)]{Low}
Lowther, G. (2010). Nondifferentiable functions of one-dimensional semimartingales. \textit{Ann. Prob.} {\bf 38} (1) 76 - 101.
\bibitem[Protter(1992)]{PRT}
Protter, P. (1992).	
\textit{Stochastic Integration and Differential Equations}. Springer-Verlag, Berlin, Heidelberg, Second edition.
\bibitem[Kechris(1995)]{Kechris1995}
Alexander~S. Kechris,
\emph{Classical Descriptive Set Theory},
Graduate Texts in Mathematics, vol.~156,
Springer, New York, 1995.
\bibitem[Bruckner(1978)]{Bruckner1978}
Bruckner, A. M. (1978).
\textit{Differentiation of Real Functions}.
Lecture Notes in Mathematics, Vol.~659.
Springer, Berlin/Heidelberg/New York.
\bibitem[Gozzi and Russo(2006)]{WeakDir}
Gozzi, F. and Russo, F. (2006).
Weak Dirichlet processes with a stochastic control perspective.
\textit{Stochastic Processes and their Applications}
{\bf 116} (11), 1563--1583.
\bibitem[Kennerberg(2026)]{CompJump}
Kennerberg, P. (2026).
Stability of Compensated Jump Integrals under Quadratic Variation Convergence.
\textit{arXiv preprint arXiv:2605.11783}.
\end {thebibliography}
\end{document}